\documentclass[a4paper,12pt]{amsart}
\usepackage{amsmath, amsthm,amssymb}

\allowdisplaybreaks

\newtheorem{thm}{Theorem}[section]

\newtheorem{lem}[thm]{Lemma}

\theoremstyle{definition}

\theoremstyle{remark}

\numberwithin{equation}{section}

\newcommand{\N}{\mathbb{N}}

\newcommand{\C}{\mathbb{C}}

\begin{document}

\author[an Huef]{Astrid an Huef}
\author[Raeburn]{Iain Raeburn}
\author[Tolich]{Ilija Tolich}
\address{Department of Mathematics and Statistics, University of Otago, PO Box 56, Dunedin 9054, New Zealand.}
\email{\{astrid, iraeburn\}@maths.otago.ac.nz, ilija.tolich@gmail.com}
\title[Star-commuting power partial isometries]{\boldmath Structure theorems for star-commuting\\ power partial isometries}



\begin{abstract}
We give a new formulation and proof of a theorem of Halmos and Wallen on the structure of power partial isometries on Hilbert space. We then use this theorem to give a structure theorem for a finite set of partial isometries which star-commute: each operator commutes with the others and with their adjoints.
\end{abstract}
\date{\today}
\maketitle

\section{Introduction}%\label{sec-intro}

The Wold-von Neumann theorem says that every isometry on a Hilbert space is a direct sum of a unitary operator and unilateral shifts. Halmos and Wallen \cite{HW} proved a similar result for \emph{power partial isometries}: operators such that all positive powers are partial isometries. Their theorem says that every power partial isometry is a direct sum of a unitary operator, some unilateral (forward) shifts, some backward shifts and some truncated shifts on finite-dimensional spaces.

There has been recurring interest in analogues of the Wold-von Neumann theorem for families of commuting isometries \cite{Su, Sl, B, P}. It has been known for many years that the most satisfactory results are those for families which star-commute, in the sense that each isometry commutes with the other isometries and with their adjoints (see \cite{Sl, BKS, Sa}, and the extensive references in \cite{Sa}). There have been similar results for pairs of star-commuting power partial isometries based on the Halmos-Wallen theorem \cite{CS,B2}.

Here we give a modern formulation of the Halmos-Wallen theorem in terms of tensor products, and use it to prove a structure theorem for finite families of star-commuting power partial isometries. This last result seems to be new, perhaps even for operators on a finite-dimensional space, and for isometries it looks quite different from the existing versions. For pairs of power partial isometries, it also looks quite different from the decomposition in \cite{B2}, and the tensor-product decompositions obtained in \cite[\S3]{CS}, which are for special cases where the individual Halmos-Wallen decompositions have a single summand, follow from our result.

\section{The Halmos-Wallen theorem}

We use the basic properties of partial isometries, as discussed in \cite[\S A.1]{R}, for example. We also need to know that if $V$ and $W$ are partial isometries, then $VW$ is a partial isometry if and only the initial projection $V^*V$ commutes with the range projection $WW^*$ \cite[Lemma~2]{HW}. 
An operator $T$ is a \emph{power partial isometry} if $T^n$ is a partial isometry for all $n\geq 0$, and then $\{T^nT^{*n}\}\cup\{T^{*n}T^n\}$ is a commuting family of projections. (We have just established a notational convention: $T^{*n}$ means $(T^*)^n$, and we allow also $T^{*n-m}$ for $(T^*)^{n-m}$ for $n\geq m$.)

Examples of power partial isometries include unitary operators, the unilateral shift $S$ on $\ell^2$, the backward shift $S^*$, and the truncated shifts $J_p$ defined in terms of the usual basis for $\C^p$ by $J_pe_n=e_{n+1}$ for $n<p$ and $J_pe_p=0$. (Notice that $p\geq 1$, and we include $J_1=0$.) The Halmos-Wallen theorem says that every power partial isometry can be constructed from these examples.

Our models involve tensor products of Hilbert spaces and bounded operators on them. All we need for the present theorem are the relatively elementary properties covered in \cite[\S2.6]{KR} and \cite[\S2.4 and \S B.1]{TFB}, for example. (Though in the next section we use some less elementary facts about tensor products of $C^*$-algebras.)

\begin{thm}[Halmos and Wallen]\label{HWthm}
Let $T$ be a power partial isometry on a Hilbert space $H$, and let $P$ and $Q$ be the orthogonal projections on $\bigcap_{n=1}^\infty T^nH$ and $\bigcap_{n=1}^\infty T^{*n}H$ respectively. Then $PQ=QP$ and the subspaces $H_u:=PQH$, $H_s:=(1-P)QH$, $H_b:=(1-Q)PH$ and 
\[
H_p:=\sum_{n=1}^p(T^{n-1}T^{*n-1}-T^{n}T^{*n})(T^{*p-n}T^{p-n}-T^{*p-n+1}T^{p-n+1})H
\] are all reducing for $T$, and satisfy $H=H_u\oplus H_s\oplus H_b\oplus\big(\bigoplus_{p=1}^\infty H_p\big)$. Further there are Hilbert spaces $M_s$, $M_b$ and $\{M_p:p\geq 1\}$ (allowing $M_*=\{0\}$) such that
\begin{enumerate}
  \item\label{1a} $T|_{H_u}$ is unitary;
  \item\label{1b} $T|_{H_s}$ is unitarily equivalent to $S\otimes 1$ on $\ell^2(\N)\otimes M_s$;
  \item\label{1c} $T|_{H_b}$ is unitarily equivalent to $S^*\otimes 1$ on $\ell^2(\N)\otimes M_b$;
  \item\label{1d} for $p\geq 1$, $T|_{H_p}$ is unitarily equivalent to $J_p\otimes 1$ on $\C^p\otimes M_p$.
\end{enumerate}
\end{thm}

This formulation of the Halmos-Wallen result was previously made informally in the proof of \cite[Theorem~1.3]{HR}. At the time, it apparently did not merit proof, though the intention in \cite{HR} was to deduce it from the original version of \cite{HW}. Here we give a direct proof, which will occupy the rest of the section. So throughout the section, $T$ is a power partial isometry.

The \emph{multiplicity spaces} $M_*$ are only unique up to isomorphism, and hence are determined by their dimension. But in our proof it is convenient to take them to be the subspaces
\begin{align*}
M_s&=(1-TT^*)QH,\\
M_b&=(1-T^*T)PH,\\
M_p&=(1-TT^*)(T^{*p-1}T^{p-1}-T^{*p}T^p)H\text{ for $p\geq 2$, and}\\
M_1&=(1-TT^*)(1-T^*T)H=\ker T\cap \ker T^*.
\end{align*}
To see that $(1-TT^*)(1-T^*T)H=\ker T\cap \ker T^*$, note that
\[
Th=0=T^*h\Longleftrightarrow T^*Th=0=TT^*h\Longleftrightarrow h=(1-TT^*)(1-T^*T)h.
\]

We begin by looking at some properties of the projections $P$ and $Q$. The projections $T^nT^{*n}$ onto the subspaces $T^nH$ form a decreasing sequence, and hence converge in the strong-operator topology to the projection $P$ onto $\bigcap_{n=1}^\infty T^nH$ (by \cite[Corollary~2.5.7]{KR}, for example). For similar reasons we also have $T^{*n}T^n\to Q$ in the strong-operator topology. Since composition is jointly strong-operator continuous on norm-bounded sets \cite[Remark~2.5.10]{KR}, we have $(T^nT^{*n})(T^{*n}T^n)\to PQ$ in the strong-operator topology. Since the range and source projections all commute with each other, this implies in particular that $PQ=QP$. Thus all the products $PQ$, $(1-P)Q$ etc. are projections, and all the subspaces $H_*$ are closed.

Next we tackle the assertion that the subspaces are reducing for $T$.

\begin{lem}
The subspace $PH=\bigcap_{n=1}^\infty T^{n}H$ is reducing for $T$ and $T|_{PH}$ is a co-isometry.
\end{lem}

\begin{proof}
Since $T:T^nH\to T^{n+1}H$, $PH$ is invariant for $T$. To see it is invariant for $T^*$, take $h\in PH$, and $n\geq 1$. Since $h\in T^{n+1}H$, we can write $h=T^{n+1}k$. Then
\[
T^*h=T^*T^{n+1}k=T^*TT^nk=T^*TT^nT^{*n}T^nk=(T^nT^{*n})(T^*T)T^nk
\]
belongs to $T^nH$. So $PH$ is reducing. To see that $T^*|_{PH}=(T|_{PH})^*$ is an isometry, note that $TT^*T^n=TT^*TT^{n-1}=TT^{n-1}=T^n$, so $TT^*$ is the identity on $T^nH$ for all $n\geq 1$.
\end{proof}

Applying this lemma to the power partial isometry $T^*$ shows that $QH$ is reducing for $T$ and that $T|_{QH}$ is an isometry. Thus the complements $(1-P)H=(PH)^\perp$ and $(1-Q)H$ are reducing, and so are the intersections $PQH=PH\cap QH$, $(1-P)QH$, $(1-Q)PH$ and $(1-P)(1-Q)H$. Now we see that both $T$ and $T^*$ are isometric on $H_u=PQH$, and hence $T|_{H_u}$ is unitary, as claimed in part~\eqref{1a}.

For the other three parts, we recall some useful identities. First, we compute
\begin{equation}\label{pull1}
T(T^nT^{*n})=(TT^*T)T^nT^{*n}=T(T^nT^{*n})(T^*T)=(T^{n+1}T^{*n+1})T,
\end{equation}
and similarly
\begin{equation}\label{pull2}
T(T^{*n}T^n)=TT^*(T^{*n-1}T^{n-1})T=(T^{*n-1}T^{n-1})(TT^*)T=(T^{*n-1}T^{n-1})T.
\end{equation}

\begin{proof}[Proof of Theorem~\ref{HWthm}\eqref{1b}]
From \eqref{pull1} and its adjoint we get
\begin{align*}\label{commute1}
T^n(1-TT^*)&=(T^nT^{*n}-T^{n+1}T^{*n+1})T^n,\quad\text{and}\\
T^{*n}(T^mT^{*m}-T^{m+1}T^{*m+1})
&=\begin{cases}
T^{*n}-T^{*n}=0&\text{if $n>m$}\\
(T^{m-n}T^{*m-n}-T^{m-n+1}T^{*m-n+1})T^{*n}&\text{if $n\leq m$.}
\end{cases}
\end{align*}
Thus with $P_n:=T^nT^{*n}-T^{n+1}T^{*n+1}$, we have $P_mP_n=\delta_{m,n}P_m$. For $h\in H$, the series $\sum_{n=0}^\infty P_nh$ telescopes and converges in norm with sum $(1-P)h$, and hence $H_s=(1-P)QH$ has a direct sum decomposition $H_s=\bigoplus_{n=0}^\infty P_nQH$.

Since $T|_{QH}$ is an isometry, we have
\[
T^{*n}T^n(1-TT^*)Q=(1-TT^*)T^{*n}T^nQ=(1-TT^*)Q.
\]
Since $T^nT^{*n}P_nQ=P_nQ$, we deduce that $T^n$ is an isomorphism of $M_s=P_0QH$ onto $P_nQH$ for all $n\geq 1$. Thus there is a unitary isomorphism $U$ of $\ell^2\otimes M_s$ onto $\bigoplus_{n=0}^\infty P_nQH=H_s$ such that $U(e_n\otimes h)=T^nh$ for $n\geq 0$ and $h\in M_s$. This isomorphism satisfies
\[
U(S\otimes 1)(e_n\otimes h)=U(e_{n+1}\otimes h)=T^{n+1}h=TU(e_n\otimes h),
\]
and hence $U(S\otimes 1)U^*=T|_{H_s}$.
\end{proof}

For part~\eqref{1c} we apply part \eqref{1b} to the power partial isometry $T^*$, and then take adjoints. It remains to prove part~\eqref{1d}. For later use, we prove a little more than we need:

\begin{lem}\label{Hp}
For $p\geq 2$, we have
\begin{equation}\label{defHp}
H_p=\bigoplus_{n=1}^p(T^{n-1}T^{*n-1}-T^{n}T^{*n})(T^{*p-n}T^{p-n}-T^{*p-n+1}T^{p-n+1})H,
\end{equation}
viewed as an internal direct sum inside $H$. Each $H_p$ is reducing for $T$, and there is a unitary isomorphism $U_p$ of $\C^p\otimes M_p$ onto $H_p$ such that $T|_{H_p}=U(J_p\otimes 1)U^*$.
\end{lem}

\begin{proof}
We set $P_1:=(1-TT^*)(T^{*p-1}T^{p-1}-T^{*p}T^p)H$, so that $M_p=P_1H$. Several applications of \eqref{pull1} and \eqref{pull2} show that for $n<p$, we have
\[
T^{n-1}P_1=(T^{n-1}T^{*n-1}-T^{n}T^{*n})(T^{*p-n}T^{p-n}-T^{*p-n+1}T^{p-n+1})T^{n-1}.
\]
We set
\[
P_n:=(T^{n-1}T^{*n-1}-T^{n}T^{*n})(T^{*p-n}T^{p-n}-T^{*p-n+1}T^{p-n+1})\quad\text{for $1\leq n\leq p$,}
\]
and then we have $P_mP_n=\delta_{m,n}P_m$ and $T^{n-1}P_1=P_nT^{n-1}$ for $1\leq n<p$. Taking adjoints gives $T^{*n-1}P_n=P_1T^{*n-1}$. Since $T^p(T^{*p-1}T^{p-1}-T^{*p}T^p)=T^p-T^p=0$, we also have $T^pP_1=0$. Thus $H_p=\sum_{n=1}^p P_nH$ is a reducing subspace for $T$. Since 
\[
T^{*n-1}T^{n-1}(T^{*n-1}T^{n-1}-T^{*n}T^n)=(T^{*n-1}T^{n-1}-T^{*n}T^n),
\]
we have $T^{*n-1}T^{n-1}P_1=P_1$. A similar calculation shows that $T^{n-1}T^{*n-1}P_n=P_n$, and hence $T^{n-1}$ is an isomorphism of $M_p=P_1H$ onto $P_nH$.

Since the $P_n$ are mutually orthogonal projections there is a unitary isomorphism $U_p$ of $\C^p\otimes M_p$ onto $H_p$ such that $U_p(e_n\otimes h)=T^{n-1}h$. Then for $n<p$, we have
\[
U_p(J_p\otimes 1)(e_n\otimes h)=U_p(e_{n+1}\otimes h)=T^nh=TU_p(e_n\otimes h),
\]
and for $n=p$ we have $U(J_p\otimes 1)(e_n\otimes h)=0$ and $TU(e_p\otimes h)=TT^{p-1}h=T^p(P_1h)=0$. Thus $U_p$ intertwines $J_p\otimes 1$ and $T|_{H_p}$, as required.
\end{proof}

\begin{proof}[End of the proof of Theorem~\ref{HWthm}]
We consider the projections
\[
Q_{m,n}:=(T^{m}T^{*m}-T^{m+1}T^{*m+1})(T^{*n}T^{n}-T^{*n+1}T^{n+1}).
\]
These projections are mutually orthogonal. The partial sums $\sum_{m=0}^M\sum_{n=0}^N Q_{m,n}$ telescope, and hence
\[
\sum_{m=0}^M\sum_{n=0}^NQ_{m,n}=(1-T^{M+1}T^{*M+1})(1-T^{*N+1}T^{N+1})
\]
converges in the strong-operator topology to $(1-P)(1-Q)$ as $M\to \infty$ and $N\to \infty$. Thus we have a direct-sum decomposition $(1-P)(1-Q)H=\bigoplus_{m,n=0}^\infty Q_{m,n}H$. The subspaces $H_p$ of Lemma~\ref{Hp} are finite direct sums of the $Q_{m,n}H$, and every $Q_{m,n}H$ is a summand of some $H_p$ --- in fact, of $H_{m+n+1}$. Thus $(1-P)(1-Q)H=\bigoplus_{p=1}^\infty H_p$. 

This completes the proof of Theorem~\ref{HWthm}.
\end{proof}

\section{Star-commuting power partial isometries}

We now consider a finite family $\{T_m:1\leq m\leq M\}$ of power partial isometries, and we assume that they \emph{star-commute}: we have $T_mT_n=T_nT_m$ for all $m,n$ and $T^*_mT_n=T_nT^*_m$ for $m\not=n$. Such families have also been described as \emph{doubly commuting}. We say that two such families $\{T_m\}\subset B(H)$ and $\{S_m\}\subset B(K)$ are \emph{simultaneously unitarily equivalent} if there is a unitary isomorphism $U$ of $H$ onto $K$ such that $UT_mU^*=S_m$ for all~$m$.

\begin{thm}\label{finitely many star-commuting decomposition}
Let $\{T_m:1\leq m\leq M\}$ be star-commuting power partial isometries on a Hilbert space $H$, and set $I=\{u,s,b\}\cup \{p\in\N:p\geq 1\}$. For each multiindex $i\in I^M$, we write $\Sigma_i:=\{m:1\leq m\leq M,\;i_m\not= u\}$, and we set $K_{i,m}=\ell^2$ if $i_m=s$ or $i_m=b$, and $K_{i,m}=\C^p$ if $i_m=p$. Then there are closed subspaces $\{H_i:i\in I^M\}$ of $H$, all of which are reducing for all the $T_m$ and which satisfy $H=\bigoplus_{i\in I^M}H_i$, Hilbert spaces $\{M_i:i\in I^M\}$, and commuting unitaries $\{V_{i,m}\in U(M_i):i_m=u\}$ such that the $\{T_m|_{H_i}:1\leq m\leq M\}$ are simultaneously unitarily equivalent to
\begin{enumerate}
\item $\big(\bigotimes_{n\in \Sigma_i}1_{K_{i,n}}\big)\otimes V_{i,m}$ if $i_m=u$;
\item $\big(\bigotimes_{n\in \Sigma_i, n\not= m}1_{K_{i,n}}\big)\otimes S\otimes 1_{M_i}$ if $i_m=s$;
\item $\big(\bigotimes_{n\in \Sigma_i, n\not= m}1_{K_{i,n}}\big)\otimes S^*\otimes 1_{M_i}$ if $i_m=b$;
\item $\big(\bigotimes_{n\in \Sigma_i, n\not= m}1_{K_{i,n}}\big)\otimes J_p\otimes 1_{M_i}$ if $i_m=p$.
\end{enumerate}
\end{thm}

\begin{proof}
We prove by induction on $M$ that the theorem holds for $M$-tuples of power partial isometries, augmented to say that all the subspaces $H_i$ are reducing for every operator $R$ that star-commutes with all the $T_m$. For $M=1$, the sets $i$ are singletons, and the subspaces $H_i$ are the subspaces $H_u$, $H_s$, $H_b$ and $H_p$ in Theorem~\ref{HWthm}. Now suppose that $R$ star-commutes with $T$. Since the projections $P$ and $Q$ in Theorem~\ref{HWthm} are strong-operator limits of the sequences $\{T^nT^{*n}\}$ and $\{T^{*n}T^n\}$, respectively, we deduce that $P$ and $Q$ star-commute with $R$. The formula~\eqref{defHp} shows that the projection onto $H_p$ is also built from range and source projections of the $T^n$, so it star-commutes with $R$ too. Thus the ranges of these projections are reducing for $R$.

So we suppose the augmented theorem is true for $M$, and that we have star-commuting partial isometries $\{T_m:1\leq m\leq M+1\}$. To simplify the formulas and reduce the number of cases, we write $S_s:=S$, $S_b:=S^*$ and $S_p:=J_p$ for $p\geq 1$. We apply the inductive hypothesis to $\{T_m:1\leq m\leq M\}$. Since we are working up to simultaneous unitary equivalence, we can conjugate by a unitary and suppose that there exist spaces $N_i$ with
\begin{align*}
H&={\textstyle\bigoplus_{i\in I^M}}H_i={\textstyle\bigoplus_{i\in I^M}}\big(\textstyle{\bigotimes_{n\in \Sigma_i}}K_{i,n}\big)\otimes N_i,
\intertext{and that for every $1\leq m\leq M$}
T_m&=\begin{cases}
{\textstyle\bigoplus_{i\in I^M}}\big(\textstyle{\bigotimes_{n\in \Sigma_i}}1_{K_{i,n}}\big)\otimes V_{i,m}
&\text{if $i_m=u$}\\
{\textstyle\bigoplus_{i\in I^M}}\big(\textstyle{\bigotimes_{n\in \Sigma_i,n\not= m}}1_{K_{i,n}}\big)\otimes S_{i_m}\otimes 1_{N_i}&\text{if $i_m\not= u$}.
\end{cases}
\end{align*}
The partial isometry $T_{M+1}$ star-commutes with all the other $T_m$, and hence by the augmentation in the induction hypothesis leaves all the summands $H_i$ invariant. Thus $T_{M+1}|_{H_i}$ star-commutes with all the operators $\big(\textstyle{\bigotimes_{n\in \Sigma_i,n\not= m}}1_{K_{i,n}}\big)\otimes S_{i_m}\otimes 1_{N_i}$ arising as summands of the $T_m|_{H_i}$. Hence it star-commutes with all the operators of the form $T\otimes 1$ for which $T$ is in the $C^*$-subalgebra of $B\big(\textstyle{\bigotimes_{j\in \Sigma_i}}K_{i,j}\big)$ generated by the operators $\big(\textstyle{\bigotimes_{n\in \Sigma_i,n\not= m}}1_{K_{i,n}}\big)\otimes S_{i_m}$.

For $i_m=p\geq 1$, the $C^*$-algebra $C^*(S_{i_m})=C^*(J_p)$ is all of $M_p(\C)$, and for $i_m=s$ or $i_m=b$, $C^*(S_{i_m})=C^*(S)$ contains the algebra of compact operators on $\ell^2$. Hence for all $m\in \Sigma_i$, the algebra $C^*(S_{i_m})$ acts irreducibly on $K_{i,m}$. Thus the spatial tensor product $\bigotimes_{m\in \Sigma_i} C^*(S_{i_m})$ acts irreducibly on $\bigotimes_{m\in \Sigma_i} K_{i,m}$, and the operator $T_{M+1}|_{H_i}$ has the form $1\otimes R_i$ for some $R_i\in B(N_i)$ (see \cite[Lemma~B.36]{TFB}, for example). Since $T_{M+1}|_{H_i}$ is a power partial isometry, so is $R_i$, and since the $T_m$ star-commute with $T_{M+1}$, the unitaries $V_{i,m}$ star-commute with $R_i$.

We now apply Theorem~\ref{HWthm} to the power partial isometry $R_i$ on $N_i$, yielding a direct sum decomposition of $N_i$. Again we can conjugate by a unitary isomorphism, and assume that
\begin{align}
\label{multiplicitydecomp}
N_i&=M_{i,u}\oplus(\ell^2\otimes M_{(i,s)})\oplus(\ell^2\otimes M_{(i,b)})\oplus \big(\textstyle{\bigoplus_{p=1}^\infty}(\C^p\otimes M_{(i,p)})\big),\quad\text{and }\\
R_i&=U_{i,u}\oplus(S\otimes 1_{M_{(i,s)}})\oplus(S^*\otimes 1_{M_{(i,b)}})\oplus \big(\textstyle{\bigoplus_{p=1}^\infty}(J_p\otimes 1_{M_{(i,p)}})\big)\notag
\end{align}
with $U_{i,u}=R_i|_{M_{(i,u)}}$ unitary. Now for $i'=(i,i_{M+1})\in I^M\times I=I^{M+1}$, we take
\begin{align}
\label{newbits}H_{i'}&=
\begin{cases}
H_i=\big(\bigotimes_{m\in \Sigma_i}K_{i,m}\big)\otimes M_{(i,u)}&\text{if $i_{M+1}=u$}\\
\big(\bigotimes_{m\in \Sigma_i}K_{i,m}\big)\otimes K_{i',{M+1}}\otimes M_{(i,i_{M+1})}&\text{if $i_{M+1}\not=u$}
\end{cases}\\
&=
\begin{cases}
\big(\bigotimes_{m\in \Sigma_{i'}}K_{i',m}\big)\otimes M_{(i,u)}&\text{if $i_{M+1}=u$ (since then $\Sigma_{i'}=\Sigma_i$)}\\
\big(\bigotimes_{m\in \Sigma_{i'}}K_{i',m}\big)\otimes M_{(i,i_{M+1})}&\text{if $i_{M+1}\not=u$.}
\end{cases}\notag
\end{align}
It follows from \eqref{multiplicitydecomp} and \eqref{newbits} that $H_i=\bigoplus_{j\in I}H_{(i,j)}$ for each $i\in I^M$, so the $H_{i'}$ give a direct-sum decomposition of $H$. Since all the $V_{i,m}$ star-commute with $R_i$, the augmentation in the case $M=1$ implies that they star-commute with $U_{i,u}$ and that all the direct summands in \eqref{multiplicitydecomp} are reducing for them; we take 
$V_{i',m}$ to be $V_{i,m}$ if $i_{M+1}=u$, and if $i_{M+1}\not=u$, we take $V_{i',m}$ to be the operator on $M_{i'}=M_{(i,i_{M+1})}$ such that $V_{i,m}|_{K_{i,m}\otimes M_{(i,i_{M+1})}}=1\otimes V_{i',m}$.

We still need to check that our subspaces are reducing for every operator $T$ which star-commutes with all the $T_m$. But  the subspaces $H_i$ are reducing for any such $T$ by the inductive hypothesis, and then as before $T|_{H_i}$ has the form $1\otimes R$ for some $R\in B(N_i)$. Since $T|_{H_i}=1\otimes R$ star-commutes with $T_{M+1}|_{H_i} =1\otimes R_i$, it follows from the case $M=1$ that the subspaces in the decomposition \eqref{multiplicitydecomp} are all reducing for $R$. This proves the augmented part of our inductive hypothesis for $M+1$, and completes our proof.
\end{proof}


\begin{thebibliography}{19}
\bibitem{B} Z.~Burdak, On decomposition of pairs of commuting isometries, \emph{Ann. Polon. Math.} \textbf{84} (2004), 121--135.

\bibitem{B2} Z.~Burdak, On a decomposition for pairs of commuting contractions, \emph{Studia Math.} \textbf{181} (2007),  33--45.

\bibitem{BKS} Z.~Burdak, M.~Kosiek, and M.~S{\l}oci{\'{n}}ski, The canonical Wold decomposition of commuting isometries with finite dimensional wandering spaces, \emph{Bull. Sci. Math.} \textbf{137} (2013), 653--658.

\bibitem{CS} X.~Catepill\'{a}n and W.~Szyma\'{n}ski, A model of a family of power partial isometries, \emph{Far East J. Math. Sci.} \textbf{4} (1996), 117--124.

\bibitem{EM} I.~Erd\'elyi and F.~Miller, Decomposition theorems for partial isometries, \emph{J. Math. Anal. Appl.} \textbf{30} (1970), 665--679.


\bibitem{HW} P.R. Halmos and L.J. Wallen, Powers of partial isometries, \emph{Indiana Univ. Math. J.} \textbf{19} (1970), 657--663.

\bibitem{HR} R. Hancock and I. Raeburn, The $C^*$-algebras of some inverse semigroups, \emph{Bull. Aust. Math. Soc.} \textbf{42} (1990), 335--348.


\bibitem{KR} R.V. Kadison and J.R. Ringrose, \textit{Fundamentals of the Theory of Operator Algebras}, vol.~I, Amer. Math. Soc., Providence, 1998.

\bibitem{P} D.~Popovici, A Wold-type decomposition for commuting isometric pairs, \emph{Proc. Amer. Math. Soc.} \textbf{132} (2004), 2303--2314.


\bibitem{R} I. Raeburn, \textit{Graph Algebras}, CBMS Regional Conference  Series in Math., vol.~103, Amer. Math. Soc., Providence, 2005.

\bibitem{TFB} I. Raeburn and D.P. Williams, \textit{Morita Equivalence and Continuous-Trace {$C\sp *$}-Algebras}, Mathematical Surveys and Monographs, vol.~60, Amer. Math. Soc., Providence, 1998.

\bibitem{Sa} J.~Sarkar, Wold decomposition for doubly commuting isometries, \emph{Linear Algebra Appl.} \textbf{445} (2014), 289--301.

\bibitem{Sl} M.~S{\l}oci{\'{n}}ski, On the Wold-type decomposition of a pair of commuting isometries,  \emph{Ann. Polon. Math.} \textbf{37} (1980), 255--262.

\bibitem{Su} I.~Suciu, On the semigroups of isometries, \emph{Studia Math.} \textbf{30} (1968), 101--110.


\end{thebibliography}
\end{document}